\documentclass[12pt]{amsart}

\usepackage{amsmath}
\usepackage{amssymb,amsfonts,latexsym}
\usepackage{geometry}
\usepackage[all]{xy}
\usepackage{xypic,color}
\usepackage{subcaption}
\usepackage{chngcntr}
\counterwithin{table}{section}

\usepackage{geometry}

\newif\iffull
\fulltrue

\newif\ifwreath
\wreathfalse

\newtheorem{thm}{Theorem}[section]
\newtheorem{cor}[thm]{Corollary}

\newtheorem{lem}[thm]{Lemma}
\newtheorem{prop}[thm]{Proposition}
\theoremstyle{definition}

\theoremstyle{remark}
\newtheorem{rem}[thm]{Remark}

\newtheorem{exam}[thm]{Example}
\numberwithin{equation}{section}

\long\def\forget#1\forgotten{}

\newcommand{\C}{\mathcal{C}}
\newcommand{\D}{\mathcal{D}}
\newcommand{\E}{\mathcal{E}}

\newcommand{\mQ}{\mathbb{Q}}

\newcommand{\mC}{\mathbb{C}}

\newcommand{\mP}{\mathbb{P}}

\newcommand{\Z}{\mathbb{Z}}

\newcommand{\Q}{\mathbb{Q}}

\newcommand{\ra}{\rightarrow}

\DeclareMathOperator\Gal{Gal}
\DeclareMathOperator{\Red}{Red}

\DeclareMathOperator\height{ht}

\def\({\left(}
\def\){\right)}

\newcommand\oline[1] {{\overline{#1}}}
\newcommand\elll{{n}}

\newcommand\lcm{{\operatorname{lcm}}}

\DeclareMathOperator\Sym{Sym}

\newcommand\gOrbs[2]{{
			\text{Orbs}_{#1}({#2})
}}

\newcommand\uos[2]{{
		{#1}^{\{#2\}}
}}
\newcommand\uosS[4]{{
		{#1}^{\{#2\}} \cup {#3}^{\{#4\}}
}}

\newcommand\orbs[2]{{
		O(\uos{#1}{#2})
}}

\newcommand\orbsS[4]{{
		O(\uosS{#1}{#2}{#3}{#4})
}}

\usepackage[colorlinks,pagebackref,pdftex, bookmarks=false]{hyperref}

\newcommand\enset[2]{ \{#1,\dots,#2\}}
\newcommand\nset[1]{\enset{1}{#1}}
\newcommand\idef[1] {{\it #1}}

\begin{document}

\title
{Symmetric Galois Groups under Specialization}%

\def\technion{Department of Mathematics, Technion - IIT, Haifa 3200, Israel}
\author{ Tali Monderer}
\address{\technion}
\email{tali.monderer@gmail.com}%
\author{ Danny Neftin}
\address{\technion}
\email{dneftin@technion.ac.il}%
\maketitle
\begin{abstract}
    Given an irreducible bivariate polynomial $f(t,x)\in \mQ[t,x]$, what groups $H$ appear as the Galois group of $f(t_0,x)$ for infinitely many $t_0\in \mQ$? How often does a group $H$ as above appear as the Galois group of $f(t_0,x)$, $t_0\in \mQ$?  We give an answer for $f$ of large $x$-degree with alternating or symmetric Galois group over $\mQ(t)$. This is done by determining the low genus subcovers of coverings $\tilde{X}\rightarrow \mP^1_{\mC}$ with alternating or symmetric monodromy groups. 
\end{abstract}

\section{Introduction}\label{sec:intro}
Let $f(t,x)\in \mQ(t)[x]$ be a polynomial with coefficients depending on a parameter $t$ and let $G$ be its Galois group. 
For all but finitely many specializations $t\mapsto t_0\in \mQ$, the Galois group $\Gal(f(t_0,x),\mQ)$ is a subgroup of $G$; and Hilbert's irreducibility theorem guarantees that the Galois group remains $G$ for infinitely many $t_0\in \mQ$. It may still hold that a proper subgroup of $G$ occurs for infinitely many $t_0\in\mQ$. For example, the polynomial $f(t,x):=x^2-t$ has a nontrivial Galois group over $\mQ(t)$, and all specializations of the form $t\mapsto q^2$ for some $q\in \mQ$  yield a rational polynomial that splits in $\mQ$. 
Given a polynomial $f(t,x)\in \mQ(t)[x]$, what subgroups of $\Gal(f(t,x),\mQ(t))$ occur as $\Gal(f(t_0,x),\mQ)$ for infinitely many rational $t_0\in \mQ$?

We are interested in the ``general" case where $G:=\Gal(f(t,x),\mQ(t))$ is a symmetric group $S_n$. Most notably it is known that: 1)  every intransitive $H\leq S_n$ for $n>5$, that occurs as $\Gal(f(t_0,x),\mQ)$ for infinitely many integral $t_0\in\Z$, must be contained in $S_{n-1}$ \cite{Mul2}; and 2)  a maximal subgroup $H\leq S_n$, for sufficiently large $n$, that occurs as $\Gal(f(t_0,x),\mQ)$ for infinitely many rational $t_0\in \mQ$ must be either $S_{n-1}$ or $S_{n-2}\times S_2$ \cite{NZ}. 
In this work we answer the above question  when $G=S_n$ for large $n$, with no maximality assumption on $H$. Our main result is the following theorem. For $g\geq 0$, let $N_g$ be the constant defined in Remark \ref{rem:same-const}.
\begin{thm}\label{thm:spec}
	Let $f(t,x)\in \Q(t)[x]$ be a polynomial with Galois group  $A_n$ or $S_n$ for  $n>N_1$. 
	Suppose $H\cong\Gal(f(t_0,x),\mQ)$ for infinitely many $t_0\in \mQ$. Then either 
	\begin{enumerate}
		\item $H=A_n$ or $S_n$; or
		\item $H=A_{n-1}$ or $S_{n-1}$; or 
		\item \label{specialization case non-generic} $A_{n-2}\lneq H \leq S_{n-2}\times S_2$.
	\end{enumerate}
	Case (3) occurs only with explicit ramification listed in Proposition \ref{prop:ramification}.
\end{thm}
Assume $\deg f=n$\footnote{Note that the theorem allows $\deg f\neq n$, that is, the Galois group may act in an arbitrary permutation representation.}. The most probable case in which $f(t_0,x)$ is reducible is case (2), where $f(t_0,x)$ factors as a product of a linear factor and an irreducible factor of degree $n-1$.
This case appears with growth rate: 
$$ \#\{t_0\in \mQ\,|\, \height(t_0)\leq N, \Gal(f(t_0,x))\cong A_{n-1}\text{ or }S_{n-1}\}  \asymp N^{2/n}, $$
where $\height$ is the natural height $\height(\frac{m}{n})=\max\{|m|,|n|\}$ for coprime $m,n\in \Z\setminus\{ 0\}$, and $f\asymp g$ for $f,g:\mathbb{N}\ra\mathbb R^+$ means that $c_1g(n)\leq f(n)\leq c_2g(n)$ holds for all $n>n_0$, where $n_0\in \mathbb N$ and $c_1,c_2\in \mathbb R$ are positive constants. 
In case (3), $f(t_0,x)$ has an irreducible factor of degree $n-2$. This is the next probable reducible case, appearing with growth $\asymp N^{4/n(n-1)}$. Case (1) is the most probable one with growth $\asymp N^2$, as the complement of cases (2), (3) and a finite set.
The growth of specializations with Galois group $H$ is inferred from the index $[G:H]$ using \cite[\S 9.7, Case 0]{Ser97}. 

Theorem \ref{thm:spec} applies to polynomials over any finitely generated field of characteristic $0$, and moreover, each of the options (1)-(3) occurs for infinitely many specializations over some number field. As case (3) happens only for specific polynomials, these are the only polynomials with Galois group $A_n$ or $S_n$ for which $f(t_0,x)$ has an irreducible factor $h\in \mQ[x]$ of degree $2\leq \deg h\leq n-2$ for infinitely many $t_0\in \mQ$.

\subsubsection*{Low genus subfields and their group-theoretic description}
The main ingredient in proving Theorem \ref{thm:spec} is classifying low genus covers with monodromy $A_n$ or $S_n$. Here, for simplicity assume $f$ is irreducible over $\mC$ and denote by $X$ the curve defined by $f$, cf.~\S\ref{sec:Hilbert} for the reducible scenario. Let $\pi:X\ra\mP^1_\mC$ be the projection to the $t$-coordinate, and $\tilde X$ be its Galois closure. Thus the Galois group $G$ acts on $\tilde X$ and $\tilde X/(G\cap S_{n-1})\cong X$, cf.~\S\ref{sec:coverings}. 

It is well known that a subgroup $H\leq G$ appears as the Galois group of $f(t_0,x)$ for infinitely many $t_0\in \mQ$ only when $\tilde X/H$ is of genus $\leq 1$, cf.~\S\ref{sec:Hilbert}. 
The maximal subgroups $H\leq G\in \{A_n,S_n\}$ for which $\tilde X/H$ is of genus $\leq 1$ were classified in \cite{GS} and \cite{NZ}. We do this for arbitrary subgroups of $A_n$ or $S_n$: 

\begin{thm}\label{thm:main-new}
	Let $g\geq 0$ and  $\pi: X\ra \mP^1_\mC$ be a covering of degree $n>N_g$, Galois closure $\tilde X$, and monodromy group $G=A_n$ or $S_n$. 
	Suppose $H\leq G$ does not contain $A_{n-1}$, and $\tilde X/H$ is of genus at most $g$.
	Then $A_{n-2}\lneq H \leq S_{n-2}\times S_2$, and the ramification of $\pi$ is listed in Proposition \ref{prop:ramification}. In fact, $\tilde X/H$ is of genus at most 1.
\end{thm}

The proof of Theorem \ref{thm:spec} is straightfoward from Theorem \ref{thm:main-new} using Faltings' theorem, see \S\ref{sec:Hilbert}. 
The main ingredient in proving  Theorem  \ref{thm:main-new} is an analysis of the transitivity of the action of $H$ on unordered sets, using results such as the Livingstone--Wagner theorem and results on multiply transitive groups. The above action is connected to the genus of $\tilde{X}/H$ by two results from the classification of primitive monodromy groups: an  inequality \cite[Lemma 2.0.13]{GS} by Guralnick--Shareshian which connects the genus of $\tilde X/H$ to the genera   $g_i$ of the quotients of $\tilde{X}$ by stabilizers of sets of cardinality $i$; and the inequalities $g_{i+1}-g_i>2$ from \cite{NZ}. 

In the case of polynomial coverings, that is, when $\pi:\mP^1_\mC\ra\mP^1_\mC$ is given by a polynomial $p\in \mC[x]$, in combination with \cite{GS} we have the following further result: 
\begin{thm}\label{thm:pols}
	Let $p:\mP^1_\mC\ra\mP^1_\mC$ be a polynomial covering of degree $n>20$, monodromy group $G=A_n$ or $S_n$, and Galois closure $\tilde X$. 
	Suppose $A_{n-1}\neq H\leq G$ is nonmaximal, and $\tilde X/H$ is of genus $0$.
	Then  $H=S_{n-2}$ and 
	$p$ is the composition of the map $\mathbb A^1_\mC\ra \mathbb A^1_\mC$, $x\mapsto x^a(x-1)^{n-a}$, for $(a,n)=1$,  with linear polynomials. 
\end{thm}
 The proof is similar to that of Theorem \ref{thm:main-new}, however instead of relying on the inequalities $g_{i+1}-g_i>2$ from \cite{NZ}, we rely on estimates from \cite{GS}. See the more general Theorem \ref{thm:pols-ff} for the genus $1$ case.

\subsubsection*{Reducible specializations}
In similarity to other results concerning the genus $0$ problem, the classification of low genus subfields given in Theorems \ref{thm:main-new} and \ref{thm:pols} is expected to have many further applications. We develop here one application which is closely related to Theorem \ref{thm:spec}. Given a polynomial $f\in \mQ(t)[x]$, it is desirable to describe the set $\Red_f$ of values $t_0\in \mQ$ where $f(t_0,x)$ is (defined and) reducible over $\mQ$. This was studied in particular by Fried \cite{Fried74, Fried86}, K\"onig \cite{K16}, Langmann \cite{L94}, M\"uller \cite{Mue,Mul2,Mul3} and others.  It is well-known that for every $f$, there exists a finite set of coverings $h_i:X\rightarrow \mP^1_{\mQ}$ such that $\Red_f$ differs by a finite set from the union of value sets $\bigcup_{i=1}^m h_i(X_i(\mQ))$. We show that when the Galois group of $f$ is $A_n$ or $S_n$ for sufficiently large $n$, the number of value sets $m$ is at most $3$ and that this upper bound is sharp. 
Let $N_1$ be the constant from Remark \ref{rem:same-const} with $g=1$. 
\begin{thm}\label{thm: HIT functions count}\label{thm:HIT}
	Let $f\in \Q(t)[x]$ be an irreducible polynomial with Galois group $A_n$ or $S_n$ for $n>N_1$. Then there exist three coverings $h_i:X_i\ra\mP^1_\mQ, i=1,2,3$ over $\mQ$ such that 
	$\Red_f$ and $\bigcup_{i=1}^3 h_i(X_i(\mQ))$ differ by a finite set. 
	
	If moreover we are not in case (3) of Theorem \ref{thm:main-new}, then $\Red_f$ differs from $h_1(X_1(\mQ))\cup h_2(X_2(\mQ))$ by a finite set, for two coverings $h_i:X_i\ra\mP^1_\mQ, i=1,2$. If furthermore $\deg f =n$, then $\Red_f$ and $h_1(X_1(\mQ))$ differ by a finite set.  
\end{thm}

Other expected future applications of Theorem \ref{thm:main-new} stem from the relation of rational (i.e., genus $0$) subfields to problems of functional decomposition. These include determining the  arithmetically indecomposable rational functions which are geometrically decomposable \cite[\S 6]{GMS}, and the Davenport-Lewis-Schinzel problem concerning the reducibility of a polynomials of the form $f(x)-g(y)\in\mC[x,y]$ \cite{KN,DKNZ}.

{\bf Acknowledgements. } We thank Michael Zieve for helpful discussions which initiated this project. The generous financial help of the Technion and through ISF grant 577/15 and BSF grant 2014173 is gratefully acknowledged.

\section{Notation and preliminaries}\label{sec:prelim}
\subsection{Orbits and stabilizers} \label{section: orbits terminology}
All actions are left actions. 
A set of cardinality $k$ is called a $k$-set. 
If $A$ and $B$ are disjoint sets, denote by $\uosS{A}{k}{B}{\ell}$ the sets with $k$ elements from $A$ and $\ell$ elements from $B$ for $k\leq |A|$, $\ell\leq |B|$.
If a group $H$ acts  on $A$ (resp.~$B$), then $H$ also acts on 
$\uos{A}{k}$ (resp.~$\uosS{A}{k}{B}{\ell}$). Denote by $\orbs{A}{k}$ (resp.~$\orbsS{A}{k}{B}{\ell}$) the number of orbits in this action.

Given a $k$-subset $A\leq \nset{n}$,  the stabilizer of $A$ in the action of $S_n$ on $k$-sets is  conjugate to $S_k\times S_{n-k} = \Sym\{1,\ldots,k\}\times \Sym\{k+1,\ldots,n\}$.
The pointwise stabilizer of $A$  is then conjugate to $S_{n-k}$. 
We  often identify the $S_n$-set of all $k$-sets  with  $S_n/(S_k\times S_{n-k})$ via the orbit stabilizer theorem. 

\subsection{Multiply Transitive Groups}\label{section:D definition}
If a permutation group $G$ on $n$ elements has a single orbit on (ordered) tuples of distinct $k$-elements from $\nset{n}$,  it is called \idef{$k$-transitive} and if it has a single orbit on unordered $k$-sets of $\nset{n}$ it is called \idef{$k$-homogenous}. Denote the number of orbits of $G$ on unordered $k$-sets by $O_k(G)$.
A theorem of Livingstone--Wagner \cite{LW} asserts that a $k$-homogenous group $G$ is $k$-transitive for $k\geq 5$; and also that $O_{k}(G)\leq O_{k+1}(G)$ for $k\leq \lfloor\frac{n}{2}\rfloor-1$. A consequence of the classification of finite simple groups is that the only $6$-transitive groups are $A_n$ and $S_n$. Without relying on this classification, the order of transitivity of a permutation group of degree $n$, other than $A_n$ or $S_n$,  is known to be bounded by a function of $n$; one such result is Babai and Seress' elementary proof that a permutation group on $n$ elements is at most $32(\log(n))^2/\log \log n$ \cite{BS}. Take $D$ to be an integer such that $D<\lfloor \frac{n}{2} \rfloor$ and a $D$-transitive group on $n$ elements must be $A_n$ or $S_n$. When assuming the classification of finite simple groups $D=6$ suffices; Otherwise, such an integer $D$ exists but we take $D$ depending on $n$, e.g.~$\lceil 32(\log(n))^2/\log \log n\rceil$.


\subsection{Function fields and ramification}\label{section: ffs, places and ramification}
A general reference on this topic is \cite{Stich}. Let $k$ be a field of characterisitic $0$ and $\oline k$ its algebraic closure. 
A \idef{function field} over $k$ is  a finite extension  of $k(t)$ where $t$ is transcendental over $k$. 

Let  $F_2/F_1$ be an extension of function fields over $\oline k$. For a place $Q$ of $F_2$ lying over a place $P$ of $F_1$  write $e(Q|P)$ for the ramification index (cf. \cite[Definition 3.1.5]{Stich}) of $Q$ over $P$.


Let $Q_1,\dots Q_r$ be the places of $F_2$ lying above  a place $P$ of $F_1$. The multiset $E_{F_2/F_1}(P):=[e(Q_1|P),\dots,e(Q_r|P)]$ is called the {\it ramification type of $P$ in $F_2$}. The place $P$ is called a branch point of $F_2/F_1$ if $e(Q_i/P)>1$ for some $i$. Letting $S$ be the set of branch points, we recall that $S$ is finite. The multiset $\{E_{F_2/F_1}(P): P\in S\}$ is called the {\it ramification type of $F_2/F_1$}. If $F_2/F_1$ is Galois with group $G$, since the inertia group of $Q_1$ over $P$ coincides with its decomposition groups over $\oline k$, the inertia group is the stabilizer in the action of $G$ on $Q_1,\ldots,Q_r$. 
To compute the ramification in composita of extensions, we use: 
\begin{lem}[Abhyankar's Lemma {\cite[Theorem 3.9.1]{Stich}, \cite[Lemma 9.2]{NZ}}]\label{lem:Abhyankars lemma} 
	Let $F_1/F$ and $F_2/F$ be function field extensions and $F_1F_2$ their compositum. Let $Q$ be a place of $F_1F_2$ that lies over places $Q_1$,$Q_2$ and $P$ in $F_1$, $F_2$ and $F$ respectively. Then
	$$e(Q|P)=\lcm(e(Q_1|P),e(Q_2|P)).$$
	If moreover, $F_1$ and $F_2$ are linearly disjoint over $K(t)$, then the number of places $Q$, over fixed $Q_1,Q_2$ as above, is $\gcd(e(Q_1|P),e(Q_2|P))$. 
\end{lem}

\subsection{Relation to coverings}\label{sec:coverings}
A general reference on this topic is \cite[Chp.~4]{Fulton}. Let $k$ be algebraically closed of characteristic $0$. 
A covering over $\pi:X\ra\mP^1_{k}$ is a morphism of (smooth projective irreducible) curves over $ k$.  It is well known (\cite[Chapter 7]{Fulton} for example) that by associating to $\pi$ the function field extension $k(X)/ k(\mP^1)$, one obtains a 1-to-1 correspondence between equivalence classes of coverings $\pi:X\ra\mP^1_k$ and isomorphism classes of function field extensions $F/ k(t)$.  In particular we define the Galois closure $\tilde X$ of $\pi$ to denote the curve corresponding to the Galois closure $\Omega$ of $k(X)/ k(\mP^1)$, equipped with an action of $G:=\Gal(\Omega/ k(\mP^1))$. Note that $\pi$ is indecomposable if and only if $k(X)/k(\mP^1)$ is minimal, that is, has no nontrivial intermediate extensions. 

For a function field $F= k(X)$, we denote by $g_F$ the genus of the curve $X$. Note that $g_F=0$ if and only if $F$ is rational, that is, $F =  k(x)$. A polynomial map $\pi:\mP^1_{ k}\ra\mP^1_{ k}$ is a covering for which $\pi^{-1}(\infty) = \{\infty\}$, that is, where $\infty$ is totally ramified in the function field extension $ k(x)/ k(t)$ corresponding to $\pi$. Such a covering is given by $y\mapsto p(y)$ for some polynomial $p\in \mC[y]$, in which case $t=p(x)$; $x$ is a root of the irreducible polynomial $p(Y)-t\in k(t)[Y]$ and hence $[k(x): k(t)]=\deg p$.  We shall  translate  Theorems \ref{thm:main-new}, and \ref{thm:pols} to function fields and restrict to this language. 

Finally note that $\pi$ can also be viewed as a topological covering. The theory of coverings then gives elements $x_1,\ldots,x_r\in G$, called {\it branch cycles}, with product $x_1\cdots x_r=1$ that generate $G$. Moreover, the branch cycles correspond to the branch points $P_1,\ldots,P_r$ of $\pi$,  so that each $x_j$ generates an inertia group $I_j$ over $P_j$, for $j=1,\ldots,r$. Since $x_j$ generates $I_j$, the cardinalities of orbits of $x_i$ (its cycle structure), coincides with $E_{F/k(t)}(P_j)$.

\subsection{Relating the genus and orbits}\label{sec:genus orbits} 

 


Let $k$ be algebraically closed, $F_0$ a function field over $k$, $P$ a place of $F_0$, and $\Omega/F_0$ a Galois extension with Galois group $G$. For a subgroup $H\leq G$, set $F:=\Omega^H$,  $n:=[G:H]=[F:F_0]$, and let $I$ be a decomposition group over $P$. It is well known (see \cite[Section 3]{GTZ} for proof) that there is a one to one correspondence between places of $F$ over $P$, and orbits of $I$ acting on $G/H$.  Moreover, if $Q$ corresponds to an orbit $O$, then $e(Q/P) = |O|$.
In particular, one has the following version of the Riemann-Hurwitz formula. Denote by $\gOrbs{I_P}{G/H}$ the orbits of $I_P$ on $G/H$. 
\begin{equation}\label{equ:RH-orbits} 
\begin{array}{rl}
2(g_{F}-1) & = \displaystyle  2n(g_{F_0}-1)+
\sum_{Q\text{ place of }F}(e(Q|Q\cap F_0)-1) \\ 
& = \displaystyle 2n(g_{F_0}-1)+\sum_{P\text{ place of }F_0}\bigl(n-\#E_{F/F_0}(P)\bigr) \\
& = \displaystyle 2n(g_{F_0}-1)+\sum_{P\text{ place of }F_0}\bigl([G:H]-\#\gOrbs{I_P}{G/H}\bigr).
\end{array}
\end{equation}

For the genus of $k$-sets, we shall use the following conclusion from the works of Guralnick--Shareshian \cite{GS} and Neftin--Zieve \cite{NZ}. 
Let $D$ be as in Section \ref{section:D definition}. 
\begin{rem}\label{rem:same-const} 
For $g\geq 0$, we choose a constant $N_g$ so that (1) and (2) below hold for every Galois extension $\Omega/k(t)$ with Galois group $G=A_n$ or $S_n$ for $n>N_g$.  Let $g_i$ denote the genus of the fixed field of the stabilizer of an $i$-set for $1\leq i\leq n/2$.
\begin{enumerate} 
\item $g_i-g_{i-1}>g$  for $i=3,\ldots,D$, and if the ramification of  $\Omega^{G\cap S_{n-1}}/k(t)$ is not in \cite[Table 4.1]{NZ} (\cite[Chapter 3]{GS}) then also for $i=2$;
\item if  $H\leq G$ is maximal and $H\neq A_n$, then either $\Omega^{H}$ is of genus $>g$, or $H$ is a point stabilizer, or the ramification of $\Omega^{G\cap S_{n-1}}/k(t)$ is in \cite[Table 4.1]{NZ} and $H$ is a $2$-set stabilizer.  
\end{enumerate}
There exists such a constant by \cite{NZ} (or by \cite{GS} in case $\Omega/k(t)$ has at least $5$ branch points). More precisely, by \cite[Theorem 3.1]{NZ} there exist constants $c,d>0$ such that 
\begin{equation}\label{equ:giNZ}
g_i-g_{i-1}>(cn-di^{15})\frac {\binom{n}{i}}{\binom{n}{2}} \text{ for }i=2,\ldots,\lfloor{n}/2\rfloor,
\end{equation} and  $N_g$  is picked so that (2) holds, and the right hand side of \eqref{equ:giNZ} is $\geq g$ for $i\leq D$.
\end{rem}
\begin{rem}\label{rem:Angenus1}
	By \cite[Corollary 2.4]{GuralnickMSRI}, a minimal extension of degree $n>N_g$ and genus at most $g$ of a function field of genus at least $1$ cannot have an alternating or symmetric Galois group. 
\end{rem}
\begin{thm}
	\label{thm:NZ orbits condition}
For every Galois extension $\Omega/k(t)$ with group $G = A_n$ or $S_n$  for  $n>N_g$, and  subgroup $H\leq G$ fixing a subfield of genus at most $g$, one has: 
	\begin{equation}\label{equ: NZ orbits equation}
	O_2(H)=O_3(H)=\dots=O_{D}(H).
	\end{equation}
	Let $F$ be the subfield of $\Omega$ fixed by a point stabilizer of $G$. If the ramification type of $F/k(t)$ is not in \cite[Table 4.1]{NZ} (\cite[Chapter 3]{GS}), then $O_1(H)=O_2(H)$ as well. 
\end{thm}

\begin{proof}[Proof of Theorem \ref{thm:NZ orbits condition}]
Set $O_i:=O_i(H)$ for all $i$. Guralnick and Shareshian \cite[Lemma 2.0.13]{GS} 
relate the genus of $\Omega^H$ to the number of $k$-orbits of the action of $H$ on $\{ 1,\dots,n \}$ as follows:
\begin{equation}\label{equ:GS sum inequality}
g_{\Omega^{H}} \geq \sum_{i=2}^{\lfloor\frac{n}{2}\rfloor}(O_i-O_{i-1})(g_{i}-g_{{i-1}})
\end{equation} 
where $g_i$ is the genus of a subfield of $\Omega$ fixed by a stabilizer of a $i$-set. 
Since $O_i-O_{i-1}\geq 0$ by the Livingstone-Wagner theorem, and since 
$g_{i}-g_{{i-1}}\geq 0$ by \cite[Lemma 2.0.12]{GS}, 
the summands of \eqref{equ:GS sum inequality} are nonnegative, and hence
\begin{equation}\label{equ: g_H bound by orbits}
g_{\Omega^H} \geq (O_i-O_{i-1})(g_{i}-g_{{i-1}})
\end{equation}
for $2\leq i \leq \lfloor \frac{n}{2} \rfloor$.
Since   $g_{\Omega^H}\leq g$ and $g_i-g_{i-1}>g$ as $n>N_g$, 
 the right hand side of \eqref{equ: g_H bound by orbits} must equal $0$ for $i=3,\ldots,D$, and hence
	$O_{2}=O_3 = \dots = O_D$. 
Similarly, if the ramification of $F/k(t)$ is not in  \cite[Table 4.1]{NZ}, then $g_2-g_1>g$ and   $O_1 = O_2$. 
\end{proof}
For polynomials,  the degree bound is further reduced to $20$ as a consequence of \cite{GS}: 
\begin{thm}	\label{thm:pols-orbits}
Let $\Omega$ be the splitting field of $f(t,x):=p(x)-t\in k(t)[x]$. Suppose $G:=\Gal(\Omega/k(t))= A_n$ or $S_n$  for  $n>20$. 
Let $H\neq A_n$ be a subgroup of $G$ that fixes a subfield of genus $0$ or $1$. 
Then $O_2(H) = O_3(H)$. 
\end{thm}	
\begin{proof}
Let $g_i$ be the genus of a subfield of $\Omega$ fixed by a stabilizer of an $i$-set. 
By \cite[Lemma 12.0.68]{GS}  we have $g_3-g_2>2$ if $n>48$, and furthermore for $n>20$ if there are at least $4$ branch points. 
The inequality is also shown for $20<n\leq 48$ in 
\cite[Theorem A.4.2]{GS} when there are at most $4$ branch points. 
As in the proof of Theorem \ref{thm:NZ orbits condition}, since $g_3-g_2>1\geq g_{\Omega^H}$,
\eqref{equ:GS sum inequality} shows that $O_2(H)=O_3(H)$. 
\end{proof}
Finally Table \ref{table:NZ1}  lists the possible ramification types of  $\Omega^{G\cap S_{n-1}}/k(t)$ 
in case $H$ is a $2$-set stabilizer and $n>20$, following  \cite[Theorem 1.2.1]{GS}. 
\begin{table}[htbp]
	\caption{
		\label{table:NZ1}
		First 9 cases of \cite[Table 4.1]{NZ}, their corresponding location in \cite[Chapter 3]{GS} and  possible monodromy group. In all entries,
		$a\in\{1,\ldots, \elll-1\}$ is odd and $(\elll,a)=1$.}
	\begin{equation*}\label{table:two-set-stabilizer}
	\begin{array}{| l | l | l | l |}
	\hline
	I1.1 & {[\elll], [a,\elll-a], \left[1^{\elll-2}, 2\right]}  & & S_{\elll} \\
	\hline
	I2.1 & {[\elll], [1^3,2^{(\elll-3)/2}], [1,2^{(\elll-1)/2}], \left[1^{\elll-2}, 2\right] } & \text{\cite[Proposition 3.0.24(e)]{GS}} & S_{\elll}  \\
	I2.2 & {[\elll], [1^2,2^{(\elll-2)/2}] \text{ twice}, \left[1^{\elll-2},2\right]} & \text{\cite[Proposition 3.0.24(c)]{GS}}
	& S_{\elll}
	\\
	I2.3 & {[\elll], \left[1^3,2^{(\elll-3)/2}\right], [2^{(\elll-3)/2},3] } & \text{\cite[Proposition 3.0.25(b)]{GS}}
	& S_{\elll}, A_{\elll}
	\\
	I2.4 & {[\elll], \left[1^2,2^{(\elll-2)/2}\right], [1,2^{(\elll-4)/2},3] } & \text{\cite[Proposition 3.0.25(d)]{GS}}
	& S_{\elll}\\
	I2.5 & {[\elll], \left[1,2^{(\elll-1)/2}\right], [1^2,2^{(\elll-5)/2},3]} & \text{\cite[Proposition 3.0.25(f)]{GS}}
	& S_{\elll}, A_{\elll}\\
	I2.6 & {[\elll], \left[1^3,2^{(\elll-3)/2}\right], [1,2^{(\elll-5)/2},4] } & \text{\cite[Proposition 3.0.25(a)]{GS}}
	& S_{\elll}, A_{\elll}
	\\
	I2.7 & {[\elll], \left[1^2,2^{(\elll-2)/2}\right], [1^2,2^{(\elll-6)/2},4] } & \text{\cite[Proposition 3.0.25(c)]{GS}}
	& S_{\elll}
	\\
	I2.8 & {[\elll], \left[1,2^{(\elll-1)/2}\right], [1^3,2^{(\elll-7)/2},4]} & \text{\cite[Proposition 3.0.25(e)]{GS}}
	& S_{\elll}, A_{\elll}\\
	\hline
	\end{array}\end{equation*}\end{table}

\subsection{Fiber Products} 
Given groups $H_1 , H_2, Q$ and homomorphisms $\varphi_i:G_i\rightarrow Q$, the fiber product $H_1\times_{Q}H_2$ is the following subgroup of $H_1\times H_2$:
\begin{equation}
\{(g_1,g_2)\in H_1\times H_2 : \varphi_1(g_1)=\varphi_2(g_2)\}
\end{equation}

Finally, we recall the following well known lemma (a.k.a.~Goursat's lemma). 
A group $Q$ is said to be a \idef{shared quotient} of $H_1$ and $H_2$ if there exist normal subgroups $N_1\vartriangleleft H_1$ and $N_2 \vartriangleleft H_2$ such that $H_1/N_1 \cong Q \cong H_2/N_2$.
\begin{lem}\label{lem:Goursat}
Let $H_1$ and $H_2$ be groups, and $H\leq H_1\times H_2$ a subgroup with surjective projections onto each of the coordinates. Then $H$ is a fiber product of the projections of $H_1$ and $H_2$ onto a shared quotient. 
\end{lem} 
	
\section{The condition $O_2(H)=O_3(H)=\dots=O_{d}(H)$}

The following proposition summarizes the results of this section. Motivated by Theorem \ref{thm:NZ orbits condition}, it determines which $H\leq S_n$, contained in a point or $2$-set stabilizer, satisfy the condition in the title. 
\begin{prop}\label{prop:groups}
Let  $n\geq 8$. Suppose  $O_2(H)=\dots =O_d(H)$ for $H\leq S_n$ and  $d\geq 3$. 
\begin{enumerate}
\item If $H$ fixes exactly one point $n$, then $H$ acts $d$-homogenously on $\{1,\ldots,n-1\}$. 
\item If $H$ stabilizes $\{n-1,n\}$, then $H$ acts $d$-homogenously on  $\{1,\ldots,n-2\}$. 
	\end{enumerate}
\end{prop}

We first deduce the following corollary. Let $D$ be the constant from Section \ref{section:D definition}. 
\begin{cor}\label{cor:groups}
Let $n\geq 8$, and suppose $H\leq S_n$ satisfies $O_2(H)=\dots =O_{D}(H)$. 
\begin{enumerate}
\item If $H$ has exactly one fixed point, 
then $H$ is $A_{n-1}$ or $S_{n-1}$.
\item If $H$ stabilizes a $2$-set, then $H$ is  $A_{n-2}$, $S_{n-2}$, $S_2\times S_{n-2}$,  $S_2\times_{C_2}S_{n-2}$ or $S_2\times A_{n-2}$.	
\end{enumerate}
\end{cor}
\begin{proof}
If $H$ has a unique fixed point, Proposition \ref{prop:groups} implies that $H$ acts $D$-homogenously on a set $U_1$ of cardinality $n-1$. Hence by definition of $D$, the projection of $H$ to $\Sym(U_1)$ contains $A_{n-1}$ and thus $H=A_{n-1}$ or $S_{n-1}$. 


Similarly, if $H$ stabilizes a $2$-set $B$, then Proposition \ref{prop:groups} implies that $H$ acts $D$-homogenously  on the complement $U_2$ of $B$. By definition of $D$, this implies that the projection $H_2$ of $H$ to $\Sym(U_2)$ is either $A_{n-2}$ or $S_{n-2}$. 

If $H$ fixes $B$ pointwise, then $H=H_2$, and we are done. 
Henceforth assume the projection $H_1$ of $H$ to $\Sym(B) = S_2$ is onto. 
By Lemma \ref{lem:Goursat}, the group $H$ is a fiber product of projections of $H_1$ and $H_2$ onto a shared quotient. The only shared quotients of $S_{n-2}$ and $S_2$ are $\{e\}$ or $S_2$, and the only shared quotient of $A_{n-2}$ and $S_2$ is $\{e\}$. Therefore in this case $H$ is $ S_{n-2}\times S_2$, $A_{n-2}\times S_2$ or $S_{n-2}\times_{C_2} S_{2}$. 
\end{proof}

We shall use the following easy two lemmas in the proof of Proposition \ref{prop:groups}.

\begin{lem}\label{lem:basic properties of $k$-orbits}
	Let $n\geq 2$ and $H\leq S_n$. Let $X,Y$ be two disjoint $H$-invariant subsets.  Denote $N_1:=|X|$ and $N_2:=|Y|$, and assume $N_1\geq N_2\geq 1$. Let $k\leq N_1+N_2$.
	Then 
	\begin{enumerate}
		\item\label{props:General orbits sum} { The orbits of $H$ on $k$-sets of $X\cup Y$ can be counted as follows:
			\begin{equation*} 
			\orbs{(X\cup Y)}{k}=\sum_{i=0}^{N_2}\orbsS{X}{k-i}{Y}{i}.
			\end{equation*}
		}
		\item\label{props:orbits estimation} {
			For $i\leq N_1$, $j\leq N_2$, the number of orbits on $\uosS{X}{i}{Y}{j}$ is bounded by:
			\begin{equation*}
			\max(\orbs{X}{i}, \orbs{Y}{j}) \leq \orbsS{X}{i}{Y}{j} \leq \min(\orbs{X}{i}\cdot |\uos{Y}{j}|,|\uos{X}{i}|\cdot \orbs{Y}{j}).
			\end{equation*}
		}
	\end{enumerate}
\end{lem}

\begin{proof}
Part \eqref{props:General orbits sum} follows by noting that the sets  $\uosS{X}{k-i}{Y}{i}$, $i=0,\ldots,N_2$ form an $H$-invariant partition of $\uos{(X\cup Y)}{k}$.

	
	
	For part \eqref{props:orbits estimation}, since $X$ and $Y$ are disjoint, we can view each orbit of $H$ on $\uosS{X}{i}{Y}{j}$ as an orbit of $H$ on the ordered tuples $(A,B)$ where $A$ is an $i$-subset of $X$ and $B$ is a $j$-subset of $Y$.  If $A_1,A_2\in \uos{X}{i}$ have distinct orbits under $H$, and $B\in \uos{Y}{j}$, then $H\cdot (A_1,B)$ and $H\cdot (A_2,B)$ are distinct orbits of $H$ on $\uosS{X}{i}{Y}{j}$. Hence $H$ has at least $\orbs{X}{i}$ orbits on $\uosS{X}{i}{Y}{j}$. By symmetry $\orbs{Y}{j}\leq \orbsS{X}{i}{Y}{j}$, giving the first inequality in \eqref{props:orbits estimation}. 
	
Since each orbit of $H$ on $\uosS{X}{i}{Y}{j}$ is of the form $H\cdot (A_1,B)$ for $A_1\in \uos{X}{i}$ and $B\in \uos{Y}{j}$,  we get that $\orbsS{X}{i}{Y}{j}\leq \orbs{X}{i}\cdot |\uos{Y}{j}|$. By symmetry  $\orbsS{X}{i}{Y}{j}\leq |\uos{X}{i}|\cdot \orbs{Y}{j}$, giving the second inequality in \eqref{props:orbits estimation}. 
\end{proof}
\begin{lem}\label{lem:2} 
Suppose $H$ acts transitively on two sets $A$ and $B$ of cardinalities $M\geq N$, respectively, and that the number of orbits of $H$ on $A\times B$ is $N$. Then there is an $H$-invariant partition of $A$ into $N$ blocks $A_1,\ldots,A_N$ so that the action of $H$ on the blocks is equivalent to its action on $B$. In particular, if $M>N$, the action of $H$ on $A$ is imprimitive, and hence it is not $2$-homogenous. 
\end{lem}	
\begin{proof}
Since $H$ acts transitively on $A$, every orbit on $A\times B$ is of length at least $M$. Since $|A\times B|=MN$, and $H$ has $N$ orbits on $A\times B$, all orbits are of length equal to $M$.
Thus, letting $H_a$ (resp.~$H_b$) denote the stabilizer of $a\in A$ (resp.~$b\in B$), 
 the index of the stabilizer  $H_a\cap H_b$ of $(a,b)$  is $M$. As $M=[H:H_a]$, we have $H_a\cap H_b = H_a$, that is, $H_b\supseteq H_a$. Identifying $A$ with $G/H_a$, the desired block system is then $G/H_b$. 
If $M>N$, the inclusion $H_a\lneq H_b$  is proper, and hence this partition is nontrivial. 
It is well known that a $2$-homogenous action is primitive. 
\end{proof}

%

%

%


Proposition \ref{prop:groups} is now a direct consequence of the following two lemmas: 
\begin{lem}\label{lem:point} Let $n\geq 7$. Suppose $H\leq S_n$ fixes pointwise a set $B\subseteq\{1,\ldots,n\}$ of cardinality $r\geq 1$, and let $A:=\{1,\ldots,n\}\setminus B$. Suppose that $ d \leq \frac{n-r}{2}$ is an integer such that one of the following holds:
\begin{enumerate}
\item $O_{r}(H) = \ldots = O_{d}(H)$ with $d > r$. 
\item $O_{r+1}(H) = \ldots = O_d(H)$ with $d > r+1$ and $H$ has no fixed points in $A$. 
\end{enumerate}
Then $H$ acts $d$-homogenously on $A$. 
\end{lem}

\begin{proof}
Set $O_k:=O_k(H)$ for  $1\leq k\leq n/2$. 
Since $B$ is fixed pointwise by $H$, we have  $\orbsS{A}{j}{B}{i}=\orbs{A}{j}\cdot \orbs{B}{i}={r \choose i} \cdot \orbs{A}{j}$ for $1\leq i \leq r$.
	Along with part \eqref{props:General orbits sum} of Lemma \ref{lem:basic properties of $k$-orbits} this gives (note that $d>r$): 
	\begin{equation}\label{equ:exp}
	\begin{array}{clll}
	O_r & =\orbs{A}{r} & +\cdots + {r\choose i}\cdot \orbs{A}{r-i} & +\cdots + r \orbs{A}{1} + 1 \\
	O_{r+1} & =\orbs{A}{r+1} & +\cdots + {r\choose i}\cdot \orbs{A}{r+1-i} & +\cdots +  r \orbs{A}{2} + \orbs{A}{1} \\
	&  & \ddots   \\
	O_{d} & =\orbs{A}{d} & +\cdots + {r\choose i}\cdot \orbs{A}{d-i} & + \cdots +r\orbs{A}{d-r+1} + \orbs{A}{d-r}.
	\end{array}
	\end{equation}
By the Livingstone--Wagner theorem,
	 $\orbs{A}{i}\leq \orbs{A}{i+1}$ for $i=1,\dots ,\lfloor  \frac{n-r}{2} \rfloor -1$. Comparing the $i$-th terms of the equalities for $O_{r+1},...,O_d$ in \eqref{equ:exp} for each $i=0,\ldots,r$, gives:
	\begin{equation}\label{equ: O_1=O_D for case I}
	\orbs{A}{1}=\orbs{A}{2}=\dots=\orbs{A}{d}.
	\end{equation}	
	If $O_r=O_{r+1}$, we get by the same comparison that $1=\orbs{A}{1}$. This gives the claim for case (1).
	Otherwise, assume that $A$ contains no fixed points. 
Let $A_1,\dots,A_s$ be the orbits of $H$ on $A$.  
Since $H$ fixes no points in $A$, we have $|A_i|\geq 2$ and thus $\orbs{A_i}{2}\neq 0$ for $i=1,\ldots,s$. Part \ref{props:General orbits sum} of Lemma \ref{lem:basic properties of $k$-orbits} yields:

\begin{equation*}\label{equ:case1}
	\orbs{A}{2}=\sum_{i=1}^s \orbs{A_i}{2} + \sum _{i<j} \orbsS{A_i}{1}{A_j}{1}\geq  s+ {s \choose 2}.
\end{equation*}
If $s>1$, this gives $\orbs{A}{2}>s$, contradicting the equality $\orbs{A}{2}=\orbs{A}{1}=s$ given by \eqref{equ: O_1=O_D for case I}. 
 Thus $s=1$, and the claim follows from \eqref{equ: O_1=O_D for case I}. 
\end{proof}

\begin{lem}\label{lem:2-set}
Suppose $\{n-1,n\}$ is an orbit of $H\leq S_n$ for $n\geq 8$,
and $O_2(H) = \ldots = O_d(H)$ for an integer $3\leq d \leq\frac{n-2}{2}$. 
Then $H$ is $d$-homogenous on $\{1,\ldots,n-2\}$.
\end{lem}
\begin{proof}	
Set $A:=\{1,\dots n-2\}$, $B:=\{n-1, n\}$, and $O_k:=O_k(H)$ for  $2\leq k\leq d-1$. 
As $O_2=O_3$ and $n-2>2$, \cite[Lemma 2.1]{BH} implies\footnote{A direct proof which does not rely on \cite{BH} appears in \cite{Mon}.} that $\orbs{A}{1}=1$.
	By part \eqref{props:General orbits sum} of Lemma \ref{lem:basic properties of $k$-orbits}: 
	\begin{equation}\label{equ:orbits sum}
	O_k= \orbs{A}{k} + \orbsS{A}{k-1}{B}{1} +\orbsS{A}{k-2}{B}{2}\text{ for $k\geq 2$}.
	\end{equation}

Since $\orbsS{A}{j}{B}{2}=\orbs{A}{j}$ for $j\geq 0$, \eqref{equ:orbits sum} for $k$ and $k+1$ gives:
	\begin{equation}\label{equ:k,k+1}
	\begin{array}{rl}
	    	O_k & = \orbs{A}{k} + \orbsS{A}{k-1}{B}{1} +\orbs{A}{k-2}, \\
	    	O_{k+1} & = \orbs{A}{k+1} + \orbsS{A}{k}{B}{1} +\orbs{A}{k-1},
	\end{array}
	\end{equation}
	where $\orbs{A}{0}=1$.	
	Since $O_k=O_{k+1}$ for $2\leq k \leq d-1$ by assumption, and since
	$ \orbs{A}{k}\leq \orbs{A}{k+1}$ and $\orbs{A}{k-2}\leq \orbs{A}{k-1}$ by the Livingstone-Wagner theorem, 
	\eqref{equ:k,k+1} gives
	\begin{equation}\label{equ:AkB1 bound}
	    \orbsS{A}{d-1}{B}{1}\leq \orbsS{A}{d-2}{B}{1}\leq \dots \leq \orbsS{A}{1}{B}{1}.
	\end{equation}
	Note that by part \eqref{props:orbits estimation} of Lemma \ref{lem:basic properties of $k$-orbits}:
	\begin{equation}\label{equ:AkB1 bound2}
	    \orbs{A}{k}\leq \orbsS{A}{k}{B}{1},
    \end{equation}
for $1\leq k \leq d-1$, and hence in combination with \eqref{equ:AkB1 bound}: 
\begin{equation}\label{equ:A1B1} \orbs{A}{k}\leq \orbsS{A}{1}{B}{1}.
\end{equation}
By part \eqref{props:orbits estimation} of Lemma \ref{lem:basic properties of $k$-orbits},
	$\orbsS{A}{1}{B}{1}\leq |B|\cdot \orbs{A}{1}\leq 2$. 

	Assume first that $\orbsS{A}{1}{B}{1}=2$. Since $A$ and $B$ are both orbits of $H$ with $|A|>|B|\geq 2$, Lemma \ref{lem:2} implies that there is an $H$-invariant partition $A=A_1\cup A_2$ such that the permutation action of $H$ on the blocks is equivalent to its action on $B$ and in particular $H$ does not act 2-homogenously on $A$. Therefore  $\orbs{A}{2}=2$ by\eqref{equ:A1B1},  and  $O_2 = 5$ by  \eqref{equ:k,k+1}.	
On the other hand, we claim that $O_3>5$, contradiciting $O_3=O_2$.  
Letting $\alpha_1,\alpha_2\in A_1$,  $\beta_1,\beta_2\in A_2$, and $b\in B$,  the sets $\{\alpha_1,\alpha_2,b\}, \{\alpha_1,\beta_1,b\}, \{\beta_1,\beta_2,b\}$ lie in different orbits of $H$ since the action on $\{A_1,A_2\}$ is equivalent to that on $B$. Thus,  $\orbsS{A}{2}{B}{1}\geq 3$. Since in addition $\orbs{A}{3}\geq \orbs{A}{2}=2$, and $\orbsS{A}{1}{B}{2} = \orbs{A}{1}=1$, we have $O_3 \geq 6$ by \eqref{equ:k,k+1}, proving the claim.

	Henceforth assume $\orbsS{A}{1}{B}{1}=1$. Thus, \eqref{equ:AkB1 bound} implies that $\orbsS{A}{k}{B}{1} = 1$ for $1\leq k\leq d-1$. 
Plugging these equalities into  $\eqref{equ:k,k+1}$ and recalling that $O_2 = \cdots = O_d$ and $\orbs{A}{d}\geq\cdots\geq \orbs{A}{1}=1$, we see that $\orbs{A}{d} = \orbs{A}{d-1} = \cdots = \orbs{A}{1}=1$, as desired. 
\end{proof}	

\begin{proof}[Proof of Proposition \ref{prop:groups}]
Part (1) is given by Lemma \ref{lem:point} with $r=1$.
Part (2) is given by Lemma \ref{lem:point} with $r=2$ in the case that $B=\{n-1,n\}$ is fixed pointwise by $H$, and by Lemma \ref{lem:2-set} in the case that $B$ is an orbit of $H$.
\end{proof}

\section{The ramification types for $A_{n-2}\leq G\leq S_{n-2}\times S_2$}
\label{section:ram criteria for fixed fields}
Let $k$ be an algebraically closed field of characteristic $0$, and $\Omega/k(t)$ a Galois extension with group $A_n$ or $S_n$. The combination of Theorem \ref{thm:NZ orbits condition} and Corollary \ref{cor:groups} gives the possibilities for groups $H\leq S_n$ with fixed field $\Omega^H$ of genus $0$ or $1$. 
The following proposition gives the possible ramification types for each such $H$. 
For $g\geq 0$, letting $N_g$ denote the constant introduced in Remark \ref{rem:same-const}, 
we recall that if $n>N_g$ and $H\leq S_{n-2}\times S_2$ has fixed field of genus $\leq g$, then 
the ramification type of $\Omega^{G\cap S_{n-1}}/k(t)$ is listed in  \cite[Table 4.1]{NZ} (or alternatively in \cite[Section 3]{GS}). 
\begin{prop}\label{prop:ramification}
	Fix $g\geq 0$, and let $\Omega/k(t)$ be a Galois extension with Galois group  $G=A_n$ or $S_n$ with  $n\geq 13$. Let $F$ be the subfield of $\Omega$ fixed by  $G\cap S_{n-1}$, 
	 and assume that the ramification type $\E$ of $F/k(t)$ is listed in \cite[Table 4.1]{NZ}. 	
	Then there exist  constants $c>0$ and $d\geq 0$ satisfying the following. If $G=A_n$:
	\begin{enumerate}
		\item \label{case:stabilizersA_n}The fields fixed by stabilizers of points or $2$-sets, i.e., $S_{n-2}\times_{C_2} S_2$ and $A_{n-1}$, are of genus $0$.
		\item The genus of the field fixed by $A_{n-2}$ is at least $\max\{2,cn-d\}$. 
	\end{enumerate}
	
	If $G=S_n$: 
	\begin{enumerate}
		\item \label{case:stabilizersS_n} The fields fixed by stabilizers of points or of $2$-sets, that is, by $S_{n-1}$ or $S_{n-2}\times S_2$, are of genus $0$.
		
		\item\label{case2-point stabilizer} The genus of the field fixed by $S_{n-2}$ is  either  $0$ or at least $\max\{2,cn-d\}$.  It is of genus $0$ if and only if $$\E=[n], [a,n-a], [2,1^{n-2}]\,\,\,\text{ (Table \ref{table:NZ1}, Type (I1.1))}.$$ 
\item\label{caseS_2A_n-2}  The field fixed by $A_{n-2}\times S_2$ is of genus $0$ (resp.~$1$) if and only if  $\E$ is in Table \ref{table:A_n-2S_2genus0 cases} (resp.~Table \ref{table:A_n-2S_2 genus 1cases}). If the genus of $\Omega^{A_{n-2}\times S_2}$ is more than $1$, then it is also at least $cn-d$. 
\begin{table}
	\caption{
		\label{table:A_n-2S_2genus0 cases}
		Entries of \cite[Table 4.1]{NZ} in which the subfield fixed by $A_{n-2}\times S_2$ is of genus $0$. Here $1\leq a\leq n-1$ is an integer coprime to $n$.}
	\begin{equation*}
	\begin{array}{|l|l|l|}
	\hline
	I2.11 & {[a,n-a],\left[2^{n/2}\right], [1^2,2^{(n-6)/2},4]  } 
	& n\equiv 2 \mod4
	\\
	I2.13 & {[a,n-a],\left[1^2,2^{(n-2)/2}\right], [2^{(n-4)/2},4] } 
	& n\equiv 0 \mod 4
	\\
	I2.15 & {[a,n-a],\left[2^{n/2}\right], [1,2^{(n-4)/2},3]  } 
	& n\equiv 2 \mod 4\\
	\hline
	
	F4.3 & {\left[1,2^{(n-1)/2}\right], [1,3^{(n-1)/3}], [3,4,6^{(n-7)/6}] } & 
	n\equiv 7 \mod 12
	\\
	
	F4.5 & {\left[1^2,2^{(n-2)/2}\right], [1,3^{(n-1)/3}], [4,6^{(n-4)/6}] } & n\equiv 4 \mod 12\\
	\hline
	\end{array}
	\end{equation*}
\end{table}
		\begin{table}
	\caption{
		\label{table:A_n-2S_2 genus 1cases}
		Entries of  \cite[Table 4.1]{NZ} in which the subfield fixed by $A_{n-2}\times S_2$ is of genus $1$.}
	\begin{equation*}
	\begin{array}{| l | l | l |}
	\hline
	I2.3 & [n],[1^3,2^{\frac{n-3}{2}}],[2^{\frac{n-3}{2}},3]  & n\equiv 1\mod 4\\
	
	I2.5 & {[n], \left[1,2^{(n-1)/2}\right], [1^2,2^{(n-5)/2},3]} &  n\equiv 3 \mod 4\\
	I2.6 & {[n], \left[1^3,2^{(n-3)/2}\right], [1,2^{(n-5)/2},4] } & 
	n\equiv 1 \mod 4
	\\
	I2.8 & {[n], \left[1,2^{(n-1)/2}\right], [1^3,2^{(n-7)/2},4]} &  n\equiv 3 \mod 4\\
	
	\hline
	F1.5 & {\left[1^2,2^{(n-5)/2},3\right], [1,2^{(n-1)/2}] \text{ thrice} } &  
	n\equiv 3 \mod 4
	\\
	
	F1.8 & {\left[1^3,2^{(n-7)/2},4\right], [1,2^{(n-1)/2}] \text{ thrice} } &  n \equiv 3 \mod 4 \\
	F1.9 & \left[2^{(n-4)/2},4\right], [1^2, 2^{(n-2)/2}]\text{ thrice}; & 	 n\equiv 0 \mod 4 \\
	\hline
	\end{array}
	\end{equation*}
\end{table}

		\item \label{caseS_2times_C_2S_n-2} The field fixed by $S_{n-2}\times_{C_2} S_2$ is of genus $0$ if and only if $$\E=[1,2^{\frac{n-1}{2}}],[1,4^{\frac{n-1}{4}}],[2,3,4^{\frac{n-5}{4}}]$$ for $n\equiv 5\mod 8$ (Type (F3.2) in \cite[Table 4.1]{NZ}). If it is not of genus $0$, it is of genus at least $\max\{2,cn-d\}$. 
	\end{enumerate}
\end{prop}

\begin{proof}
We use Magma to carry out the following algorithm on $\E$. A computer free proof appears in \cite{Mon}. \\

{\bf Notation and assumptions:} View $\E$ as a set of conjugacy classes of $S_n$ or $A_n$ (i.e., partitions of $n$), each corresponding to (the conjugacy class of a decomposition group of) a single place $P$ of $k(t)$. Denote $\E=\{ \E_P\}$ where $P$ runs over the ramified places of $k(t)$. By  \cite[\S 4]{NZ},  the fields fixed by $1$-point and $2$-set stabilizers for ramification types appearing in  \cite[Table 4.1]{NZ} are of genus $0$. This gives Case \ref{case:stabilizersA_n} for $G=A_n$ and $G=S_n$. 
\\
{\bf Algorithm:} 
\\
\noindent\textbf{Step 0:} {\it Determine if $G$ is alternating or symmetric. If symmetric, calculate the genus of $\Omega^{A_n}$.}
To do this, count the number $s$ of ramification types $\E_P\in \E$ that correspond to an odd permutation. 

If $s=0$, then $G=A_n$, otherwise $G=S_n$ (as $\E$ denotes the conjugacy classes of a generating set of $G$). In the latter case, 
due to the correspondence between ramification and orbits of decomposition groups described in \S\ref{sec:genus orbits}, and since the number of orbits of an element of $S_n$ on $S_n/A_n$ corresponds to its parity, the number $s$ also gives the Riemann-Hurwitz contribution of the extension $\Omega^{A_n}/k(t)$. Thus, we calculate the genus of $\Omega^{A_n}$ using the Riemann-Hurwitz formula.
\\
\noindent\textbf{Step I:} {\it In case $G=S_n$, find the ramification  type  $\E'$ of $\Omega^{A_{n-1}}/\Omega^{A_n}$. }

To form $\E'$, run the following procedure:

{\bf Procedure I:}
For each branch point $P$ of $F/k(t)$ with corresponding ramification $\E_P\in \E$, do:
\begin{enumerate}
\item If $\E_P$ corresponds to an even permutation, include it twice into $\E'$;
\item If $\E_P$ corresponds to an odd permutation type, construct and include the multiset $\E'_P$ in $\E'$. To construct $\E'_P$,  for each $r\in \E_P$ do:
\begin{itemize}
\item if $r$ is even, include $r/2$ twice in $\E'_P$; 
\item if $r$ is odd, include $r$ once in $\E'_P$. 
\end{itemize}
\end{enumerate}

(In Figure \ref{diag:subgroups}, Procedure I is applied to the yellow line in order to compute the ramification of the red line above it.)

\textbf{Validity of Procedure I:}
We claim that the ramification type $\E'$ is indeed the ramification type of $\Omega^{A_{n-1}}/\Omega^{A_n}$.
Since $\Omega^{A_n}$ and $F=\Omega^{S_{n-1}}$ are linearly disjoint, this follows from Abhyankar's lemma:
For each branch point $P$ of $F/k(t)$, if $\E_P$ corresponds to an even permutation, then the place $P$ splits in the quadratic extension $\Omega^{A_n}/k(t)$. Hence $\Omega^{A_n}$ has two places $Q_1$ lying over it, and by Abhyankar's lemma  $E_{\Omega^{A_{n-1}}/\Omega^{A_n}}(Q_1)$ is the same as $E_{\Omega^{S_{n-1}}/k(t)}(P)$ for both possibilities for $Q_1$.
If $\E_P$ corresponds to an odd permutation, then there is single place $Q_1$ of $\Omega^{A_n}$ lying over $P$ with $e(Q_1/P)=2$. 
 Abhyankar's lemma then implies that for every place $Q_2$ of $\Omega^{S_{n-1}}$, there is either a unique place $Q$ of $\Omega^{A_{n-1}}$ lying over both $Q_1$ and $Q_2$ if  $e:=e(Q_2/P)$  is odd, or there are two such places $Q$ if $e$ is even. In the former case,  $e(Q/Q_1) = e$ for the unique place $Q$ lying over $Q_1$, and in the latter case  $e(Q/Q_1) = e/2$ for both places $Q$ lying above $Q_1$, proving the claim. 
 \\
 

 	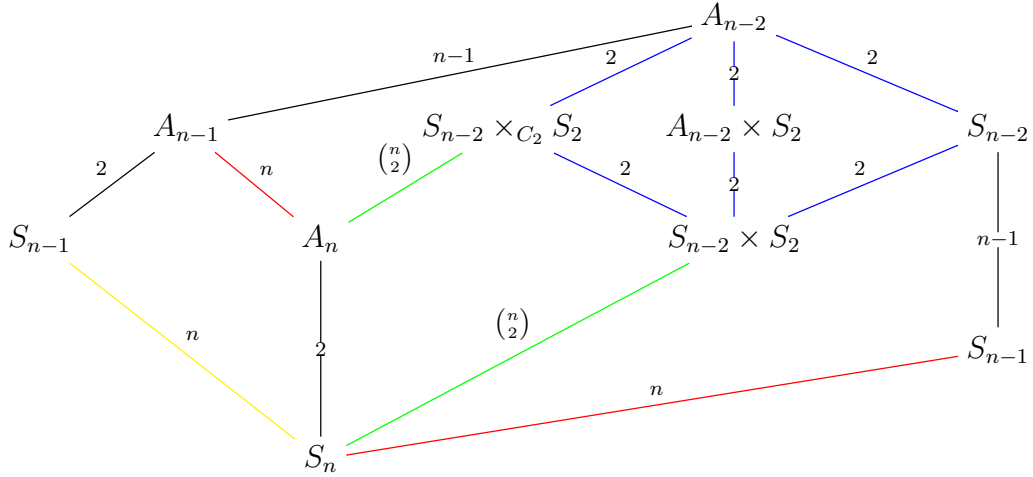
\begin{figure}\caption{Candidate genus $0$ and $1$ subgroups of $S_n$}\label{diag:subgroups}
	\begin{displaymath}
	\xymatrix {
		&											&	&																	&  {\ar@{-}_{n-1}[dlll] } \ar@{-}_{2}@[blue][dl] \ar@{-}|-{2}@[blue][d]	A_{n-2}\ar@{-}^{2}@[blue][drr]	&	& \\
		& \ar@{-}_{2}[dl] A_{n-1} \ar@{-}^{n}@[red][dr] &	  & \ar@{-}_{\binom{n}{2}}@[green][dl] S_{n-2}\times_{C_2}S_2  \ar@{-}^{2}@[blue][dr]  & A_{n-2}\times S_2 \ar@{-}|-{2}@[blue][d]& & \ar@{-}_{2}@[blue][dll] S_{n-2} \ar@{-}|-{n-1}[dd]  \\
		S_{n-1} \ar@{-}^{n}@[yellow][ddrr]& & A_n \ar@{-}|-{2}[dd] & &S_{n-2}\times S_2 \ar@{-}_{\binom{n}{2}}@[green][ddll] & & \\
		& &     & & & &  \ar@{-}_{n}@[red][dllll] S_{n-1} \\
		& & S_n &  & & &\\
	}
	\end{displaymath}
\end{figure}

\noindent {\bf Step II:} {\it If $G=S_n$, find the genus and ramification of $2$-set stabilizers in $A_n$ and in $S_n$. } 
The following procedure takes the cycle structure $\C_P$ of an element $x_P\in S_n$  (i.e., a partition of $n$, or the multiset of cardinalities of the orbits of $x_P$ on $\{1,\ldots,n\}$) and gives a partition $\C_{2,P}$ of $\binom{n}{2}$ which represents the cardinalities of orbits of $x_P$ on $2$-sets of $\{1,\ldots,n\}$ (i.e., the cycle structure of $x_P$ as an element of $S_{\binom{n}{2}}$). See \cite[Lemma 4.1]{NZ} for validity.


\textbf{Procedure II:}
Given a partition $\C_P$ of $n$, construct a partition $\E_{C,P}$ of $\binom{n}{2}$ as follows:
\begin{enumerate}
\item For every two entries $r_1,r_2\in \C_P$, add $\gcd(r_1,r_2)$ copies of $\lcm(r_1,r_2)$ to $\C_{2,P}$;
\item For each $r\in \E'_P$:
\begin{itemize}
\item For every even entry $r\in \C_P$, add $r/2$ copies of $r$ and a single copy of $r/2$ to $\C_{2,P}$; 
\item For every odd entry $r\in \C_P$, add $(r-1)/2$ copies of  $r$ to  $\C_{2,P}$. 
\end{itemize}
\end{enumerate}
Recall from \S\ref{sec:genus orbits} that the ramification of a place $P$ of $k(t)$ in the extension $\Omega^{S_{n-2}\times S_2}$ corresponds to the multiset of cardinalities of orbits of the decomposition group of $P$ on $S_n/S_{n-2}\times S_2$, which in turn is equivalent to the orbits of the decomposition group of $P$ on $2$-sets of $\{1,\ldots,n\}$. Thus apply Procedure II to all elements of $\E$ to find the ramification type $\E_2$ of $\Omega^{S_{n-2}\times S_2}/k(t)$. Similarly, apply Procedure II to all elements of the multiset $\E'$ found in the previous step in order to find the ramification type $\E'_2$ of $\Omega^{S_{n-2}\times_{C_2}S_2}/\Omega^{A_n}$.

(In Figure \ref{diag:subgroups}, Procedure II is applied to the red lines in order to compute the ramification of the green lines.)
Afterwards, use the Riemann--Hurwitz formula for $\Omega^{S_{n-2}\times_{C_2} S_2}/\Omega^{A_n}$ to find the genus of $\Omega^{S_{n-2}\times_{C_2}S_2}$. (The ramification type $\E'$ gives the Riemann-Hurwitz contribution of this extension, and the genus of the field fixed by $A_n$ was found in the previous step. )
\\


\noindent {\bf Step III:} {\it In case $G=S_n$, find the genus of the $2$-point stabilizers $A_{n-2}$ and $S_{n-2}$.} 
Calculate the Riemann-Hurwitz contribution in the extensions $\Omega^{S_{n-2}}/\Omega^{S_{n-2}\times S_2}$ and $\Omega^{A_{n-2}}/\Omega^{S_{n-2}\times_{C_2} S_2}$ using the following procedure described in \cite[Proposition 5.1]{NZ}:

\textbf{Procedure III}:
\begin{enumerate}
	\item Count the total number of even entries in $\E$.
	\item Count the total number of even entries in $\E'$.
\end{enumerate}

Afterwards, use the Riemann-Hurwitz formula to calculate the genera of $\Omega^{A_{n-2}}$ and $\Omega^{S_{n-2}}$.

\noindent {\bf Step IV:} {\it In case $G=S_n$, Find the genus of $A_{n-2}\times S_2$.} 
Let $g_1,g_2,g_3$ denote the genera of $\Omega^{S_{n-2}},\Omega^{S_{n-2}\times_{C_2}S_2}$, and $ \Omega^{A_{n-2}\times S_2}$, respectively. Denote by $g_0$ and $g'$ the genera of $\Omega^{S_{n-2}\times S_2}$ and $\Omega^{A_{n-2}\times S_2}$. A formula for  the genera of intermediate extensions of a biquadratic extension is given in \cite{Acc}:
\begin{equation}\label{equ:Accg3}
g_3 = g'-g_1-g_2+2g_{0}
\end{equation}
The formula is applicable as $\Omega^{S_{n-2}}$,  $\Omega^{S_{n-2}\times_{C_2}S_2}$ and $ \Omega^{A_{n-2}\times S_2}$ are the three quadratic intermediate extensions of the biquadratic extension $\Omega^{A_{n-2}}/\Omega^{S_{n-2}\times S_2}$. (In figure \ref{diag:subgroups}, this biquadratic extension is denoted by the blue lines).\\

\noindent{\bf Step V: }{\it If $G=A_n$, calculate genus of $\Omega^{A_{n-2}}$.} In this case, $A_{n-2}$ is a two-point stabilizer of $G$ and so as in Step III, calculate its genus using the Riemann Hurwitz formula for $\Omega^{A_{n-2}}/\Omega^{S_{n-2}}$, where the Riemann-Hurwitz contribution is given by counting the total number of even entries in $\E$.
\end{proof}


\section{Proofs of Theorems \ref{thm:main-new} and \ref{thm:pols}}
The following theorem is the function field version of Theorem \ref{thm:main-new}. Let $k$ be an algebraically closed field of characteristic $0$. Let $N_g$ be as in Remark \ref{rem:same-const}.
\begin{thm}\label{thm:ff-main-new}
	Fix $g\geq 0$. Let $F/k(t)$ be a minimal extension  of degree $n>N_g$ with Galois closure $\Omega$, and assume $G:=\Gal(\Omega/k(t))$ is $A_n$ or $S_n$. Suppose $\Omega^H$ is of genus $\leq g$ for some $H\leq G$, such that $H\not\geq A_{n-1}$.
	Then $A_{n-2}\lneq H \leq S_{n-2}\times S_2$, and the ramification of $F/k(x)$ is listed in Proposition \ref{prop:ramification}. In fact, $\Omega^H$ is of genus $\leq 1$.
\end{thm}
\begin{proof}
 If $H\subseteq A_n$, we may assume that the genus of the field fixed by $A_n$ is $0$ (otherwise the field  $\Omega^{A_n}$ is of genus $1$, which is impossible when $n>N_g$, see Remark \ref{rem:Angenus1}), and thus  replace $G$ by $A_n$. Let $M$ be the maximal subgroup of $G$ containing $H$. (In particular, $M\neq A_n$). Let $N_g$ be the constant from Remark \ref{rem:same-const}.  In particular for all $n>N_g$, as in Remark \ref{rem:same-const}, either the genus of $\Omega^{M}$ is strictly larger than $g$ or $M$ is a point stabilizer or a $2$-set stabilizer of $G$ by \cite[Theorem 1.1]{NZ}. Thus $H$ is contained in a point stabilizer or a $2$-set stabilizer of $G$. 
Due to the genus assumption on $\Omega^H$, Theorem \ref{thm:NZ orbits condition} implies that $O_2(H)=\dots =O_{D}(H)$. If furthermore the ramification type of $F/k(t)$ is not one of the exceptions given in  \cite[Table 4.1]{NZ}, then we also get $O_1(H)=O_2(H)$. Corollary \ref{cor:groups} therefore gives the list of possibilities for  $H$. Proposition \ref{prop:ramification} then gives the occuring ramification types for each possibility for $H$, and also  implies that since the genus of $\Omega^{H}$ is a constant not depending on $n$, it is less than or equal to $1$.
\end{proof} 

To prove  Theorem \ref{thm:pols}, we first verify that there are $5$ branch points:
\begin{lem}\label{lem:5br-pts}
Let $f(t,x) = p(x)-t\in k(t)[x]$ be a polynomial with  splitting field $\Omega$, and Galois group $G=A_n$ or $S_n$ for $n\geq 17$. Let  $G_1\leq G$ be the stabilizer of a root $x_1$, and $G_2$ the stabilizer of a $2$-set which does not contain  $x_1$.  Then:
\begin{enumerate}
\item The extension $\Omega/k(x_1)$ has at least $5$ branch points;
\item If the ramification of $k(x_1)/k(t)$ is of one of the types (I1.1)-(I2.8) in Table \ref{table:NZ1}, then $\Omega^{G_2}(x_1)/\Omega^{G_2}$ has at least $5$ branch points. 
\end{enumerate}
\end{lem}
See Appendix \ref{sec:appendix} for the proof. We next prove a strengthening of Theorem \ref{thm:pols}:
\begin{thm}\label{thm:pols-ff}
	Let $f(t,x)=p(x)-t\in k(t)[x]$ be a polynomial of degree $n>20$, splitting field $\Omega$ and Galois group $G=A_n$ or $S_n$. 
	Suppose $\Omega^H$ is of genus at most $1$ for nonmaximal $H\leq G$ which does not contain $A_{n-1}$. 
	Then   $H = S_{n-2}$ or $A_{n-2}\times S_2$. 

	Furthermore, if $H=S_{n-2}$, then the genus of $\Omega^H$ is $0$, and up to composition with linear polynomials $p$ equals $x^a(x-1)^{n-a}$ for some $1\leq a<n$ coprime to $n$. If $H=A_{n-2}\times S_2$, then the genus of $\Omega^H$ is $1$, and the ramification of the polynomial covering $p$ is one of types (I2.3), (I2.5), (I2.6), (I2.8) in Table \ref{table:A_n-2S_2 genus 1cases}. 
	\end{thm}
\begin{proof}
Let $G$ act on the set $\{1,\ldots,n\}$. As in the proof of Theorem \ref{thm:ff-main-new},
if $H\subseteq A_n$, we  replace $G$ by $A_n$. Let $M$ be the maximal subgroup of $G$ containing $H$. (In particular, $M\neq A_n$). 
By \cite[Theorem 1.2.1]{GS}, $M$ is a point   or  $2$-set stabilizer. 
Thus Theorem \ref{thm:pols-orbits} implies that $O_2(H)=O_3(H)$. 
If $H$ has only one fixed point, then $H$ is $3$-homogenous on a subset $U_1$ of cardinality $n-1$ by  Proposition \ref{prop:groups}.(1). 
If on the other hand $H$ stabilizes a $2$-set, then 
$H$ is $3$-homogenous on a subset $U_2$ of cardinality $n-2$ by Proposition \ref{prop:groups}.(2). Henceforth fix $i\in \{1,2\}$ such that $H$  is $3$-homogenous on $U_i$, and let $G_i\supseteq H$ be the stabilizer of an $i$-set. Let $\oline H$ and $\oline G_i$  be the images of $H$ and $G_i$ under the projection $\pi_i:G_i\ra \Sym(U_i)\cong S_{n-i}$, respectively. Let $\oline\Omega = \Omega^{\ker \pi_i}$ so that $\Gal(\oline\Omega/\Omega^{G_i})\cong \oline G_i\cong A_{n-i}$ or $S_{n-i}$. 

Letting $V\leq \oline G_i$ be a stabilizer of a point in $U_i$,  Lemma \ref{lem:5br-pts} implies\footnote{Note that if $i=2$, then the ramification of $\Omega^{G_1}/k(t)$ is  in Table \ref{table:NZ1} as in Remark \ref{rem:same-const}, and hence the lemma can be applied in that case.} that $\oline\Omega^{V}/\Omega^{G_i}$ has at least $5$ branch points. 
If the core of $\oline H$ in $\oline G_i$ is trivial, then $\oline\Omega^{H}/\Omega^{G_i}$  also has at least $5$ branch points. 
However, since the latter  is an extension of genus $\leq 1$ with $3$-homogenous stabilizer $\oline H$ on $U_i$, \cite[Theorem 1.1.2]{GS} implies that  $\oline H\cong A_{n-i}$ or $S_{n-i}$.

Since $H$ does not contain $A_{n-1}$, as in the proof of Corollary \ref{cor:groups} this implies that $H$ is one of the groups  $A_{n-2}$, $S_{n-2}$, $S_2\times S_{n-2}$,  $S_2\times_{C_2}S_{n-2}$ or $S_2\times A_{n-2}$. 
The corresponding ramification types are then given by Proposition \ref{prop:ramification}. The only resulting ramification types with an $n$-cycle are (I1.1) with $H=S_{n-2}$, or (I2.3), (I2.5), (I2.6) and (I2.8) with $H=A_{n-2}\times S_2$. In case the ramification is (I1.1), by composing with linears we may assume the branch point of type $[a,n-a]$ is $0$, and its preimages under $p$ are $0$ and $1$, that is, $p(x)$ is a constant multiple  of $x^a(x-1)^{n-a}$. 
\end{proof}

\section{Hilbert irreducibility}\label{sec:Hilbert}

Let $k$ be a finitely generated field of characteristic $0$.
Let $f\in k(t)[x]$ be a polynomial with splitting field $\Omega$ and Galois group $A$. 
For a place $(t-t_0)\lhd k[t]$, 
let $D_{P}\leq A$ denote the decomposition group of a prime $P$ of the integral closure of $k[t]$ in $\Omega$ which lies over $(t-t_0)$. 
Note that by varying $P$ over the places of $\Omega$ lying above $(t-t_0)$, 
we obtain the conjugates of $D_{P}$ in $A$.  We denote by $D_{t_0}$ the conjugacy class of such subgroups. For $D\leq A$, we write $D=D_{t_0}$ to denote that $D$ is some conjugate of $D_{P}$.
For every $t_0\in k$ which is not a root of the discriminant $\delta_f\in k(t)$ of $f$, it is well known that  $\Gal(f(t_0,x),k)$ is permutation isomorphic to $D_{P}$ \cite[Lemma 2]{KN2}. 


 The following well known proposition describes the relevant properties of $D_{t_0}$, cf.~\cite[Prop.~2.4]{KN}. Let $\tilde{X}$ be the (irreducible smooth projective) curve corresponding to $\Omega$.  If $D$ is the decomposition group at an unramified place $t\mapsto t_0\in \mQ$, then there exists a natural covering $f_D:X_D\ra\mP^1_\mQ$ from the quotient $X_D:=\tilde X/D$.
\begin{prop}
\label{prop:spec}
Let $f\in k(t)[x]$ be irreducible with 
Galois groups $G$ and $A$ over $\oline k(t)$ and $k(t)$, respectively. 
Suppose $t_0\in k$ is neither a root nor a pole of $\delta_f(t)$, 
and $D=D_{t_0}$ is its decomposition group. 
Then:
\begin{enumerate} 
\item 
$t_0\in f_D(X_D(k))$, and $DG = A$; 
\item $f(t_0,x)\in k[x]$ is reducible if and only if $D$ is intransitive. 
\end{enumerate}
\end{prop}

As a corollary to Theorem \ref{thm:main-new} we therefore have the following strengthening of Theorem \ref{thm:spec}. Let $N_1$ be the constant from Remark \ref{rem:same-const} for $g=1$. 
\begin{thm}\label{thm:spec-fld}
Let $f(t,x)\in k(t)[x]$ be a polynomial with Galois group $A=A_n$ or $S_n$ over $k(t)$ for $n>N_1$.  
If $D\leq A$ appears as the Galois group of $f(t_0,x)\in k[X]$ for infinitely many $t_0\in k$, then 
either $A_{n-1}\leq D\leq S_n$, or 
$A_{n-2}\lneq D \leq (S_{n-2}\times S_2)$ and the ramification of the fixed field $k(y)$ of $A\cap S_{n-1}$ is listed in Proposition \ref{prop:ramification}.
\end{thm}
\begin{proof}
By Proposition \ref{prop:spec}, if $D = \Gal(f(t_0,x),\mQ)$ for infinitely many $t_0\in \mQ$, then 
$X_D(k)$ is infinite. As in addition $X_D$ is geometrically irreducible, $X_D$ is of genus $\leq 1$ by Faltings' theorem.  
Letting $G$ be the Galois group of $f$ over $\oline k(t)$ and
 setting $C:=D\cap G$,  Theorem \ref{thm:main-new} therefore implies that either $A_{n-1}\leq C\leq S_n$ or $A_{n-2}\lneq C\leq S_{n-2}\times S_2$ and the ramification of $f_{A_1}$, for a point stabilizer $A_1$, is described by Proposition \ref{prop:ramification}. It therefore follows that $D$ is also of the required form. 
\end{proof}

The following is a well known corollary to Proposition \ref{prop:spec}.
\begin{cor}(\cite[Corollary 2.5]{KN})
\label{cor:kronecker}\label{cor:faltings}
Let $f(t,x)\in k(t)[x]$ be an irreducible polynomial with 
Galois groups $A$ and $G$ over $k(t)$ and $\oline k(t)$, respectively. 
Then $\Red_f$ and $\bigcup_D f_D(X_D(k))$
differ by a finite set, where $D$ runs over maximal intransitive subgroups of $A$ for which $X_D$ is of genus $\leq 1$ and $DG =A$. 
\end{cor} 

We can now deduce Theorem \ref{thm:HIT}: 
\begin{proof}[Proof of Theorem \ref{thm:HIT}]
Let $A$ and $G$ be the Galois groups of $f$ over $k(t)$ and $\oline k(t)$, respectively. 
 By Corollary  \ref{cor:faltings}, the set $\Red_f$ and the union
 $\bigcup_{D} f_D(X_D(k))$ 
 differ by a finite set, where $D$ runs over the set $\D$ of (conjugacy classes of) maximal intransitive subgroups $D\leq A$ for which $DG=A$ and $X_D$ is of genus $\leq 1$. 
As in the proof of Theorem \ref{thm:spec-fld}, $C:=D\cap G$ and $D$ are either intermediate subgroups between $A_{n-1}$ and $S_n$ or  intermediate subgroups between $A_{n-2}$ and $S_{n-2}\times S_2$ different from $A_{n-2}$. In the latter case, the ramification of $f_{A\cap S_{n-1}}$ is in  \cite[Table 4.1]{NZ}. Since by Proposition \ref{prop:ramification}, at most one of curves $X_{D}$ is of genus $\leq 1$ for $D\in \{S_{n-2}, S_{n-2}\times_{S_2} S_2, A_{n-2}\times S_2\}$, the largest subset of $\mathcal D$ consisting of (conjugacy classes) subgroups $D$ for which $C=D\cap G\in  \{S_n,A_n,S_{n-1},A_{n-1},S_{n-2}\times S_2, S_{n-2}\times_{S_2} S_2, A_{n-2}\times S_2\}$, and in which no group contains the other, is of cardinality $3$, cf.~Figure \ref{diag:subgroups}\footnote{To obtains $3$ such groups $D$, one can pick $\mathcal D = \{A_n, S_{n-1}, S_{n-2}\times S_2\}$.}.
\end{proof}
The following example shows that the three value sets in Corollary \ref{cor:faltings} is a sharp bound.
\begin{exam}
 Let $\Omega$ be the splitting field of $x^a(x-1)^{n-a}-t\in \mQ(t)[x]$ so that $\Gal(\Omega/\mQ)=S_n$ for $n>N_1$. 
Let $f(t,x)\in \mQ(t)[x]$ be  the minimal polynomial of a primitive element for $\Omega$. Letting $\mathcal D = \{S_{n-1},S_{n-2}\times S_2,A_n\}$, the fixed fields $\Omega^D$, $D\in \mathcal D$ are of genus $0$, and moreover $\mathcal D$  is the set of maximal subgroups of $S_n$ with fixed field of genus $\leq 1$. 
Since the action of $\Gal(f(t,x),\mQ)$ is regular,  every  $D\in \mathcal D$ is intransitive. Thus, in this case $Red_f$ is the union of three values sets and a finite set by Corollary \ref{cor:faltings}. 
\end{exam}

\appendix
\section{Proof of Lemma \ref{lem:5br-pts}}\label{sec:appendix}
Let $P\neq \infty$ be a branch point of $\Omega/k(t)$, 
let $Q_1,\ldots,Q_r$ be the places  of $k(x_1)$ over $P$, and $e_i = e_{k(x_1)/k(t)}(Q_i/P)$, $i=1,\ldots,r$ their ramification indices. 
Let $F := \Omega^{\hat G_2}$ for a two point stabilizer $\hat G_2\leq G_1$, so that $F\supset k(x_1)$. 
We shall use the following version of Abhyankar's lemma \cite[Lemma 5.4]{NZ} (with $t=2$). For every pair of distinct places $Q_i, Q_j$, there is a place $Q_{i,j}$ of $F$ which lies over $Q_i$, and has ramification index $\lcm(e_i,e_j)/e_i = e_j/\gcd(e_i,e_j)$. In particular, since $e_j>1$ for some $j$, every place $Q_i$ with $e_i =1$ is a branch point.

Consider the finite branch points $P_1,\ldots,P_s$ of $k(x_1)/k(t)$. 
Since $G\in\{A_n,S_n\}$ is noncyclic and generated by $s$ branch cycles, we have $s>1$. 
Letting $r_i := \#E_{k(x_1)/k(t)}(P_i)$, the Riemann--Hurwitz formula gives: 
\begin{equation}\label{equ:RH-pol}
n-1 = \sum_{i=1}^s (n-r_i).
\end{equation}
Also let $u_i = \#\{e\in E_{k(x_1)/k(t)}(P_i)\,|\,e=1\}$ and $v_i =\#\{e\in E_{k(x_1)/k(t)}(P_i)\,|\,e> 1\}.$ 

If $s\geq 3$, \eqref{equ:RH-pol} gives $\sum_{i=1}^s r_i = (s-1)n+1$. 
Since $u_i + 2v_i \leq n$, we have $v_i\leq (n-u_i)/2$. Thus $r_i = u_i +v_i\leq (n+u_i)/2$. 
In combination with \eqref{equ:RH-pol} this gives 
$\sum_{i=1}^s(n+u_i)/2 \geq (s-1)n+1$ or $\sum_{i=1}^su_i\geq (s-2)n+2$. In particular by the above Abhyankar lemma, $F/k(x_1)$ has  at least $\sum_{i=1}^s u_i\geq n+2> 5$ branch points.

It remains to consider the case $s=2$. 
We first claim that if the ramification of $k(x_1)/k(t)$ is not one of the types (I2.3)-(I2.8) in Table \ref{table:NZ1}, then $\Omega/\Omega^{G_1}$ has at least $5$ branch points. The proof of this claim is similar to the proof of case I2 of \cite[Proposition 10.1]{NZ} with $u=1$. We give the details for completeness: 
If $u_1+u_2\geq 5$, the conclusion follows from the above Abhyankar lemma. Henceforth assume $2\leq u_1+u_2\leq 4$. By the Riemann--Hurwitz formula one has:
$$ n-1 = \sum_{i=1}^2\sum_{e\in E_i}(e-1)  = \sum_{i=1}^2\left(v_i + \sum_{e\in E_i, e\geq 3}(e-2)\right), $$
where $E_i:=E_{k(x_1)/k(t)}(P_i)$, $i=1,2$. Since $v_i = r_i-u_i$ and $r_1+r_2 = n+1$ by \eqref{equ:RH-pol}, 
\begin{equation}\label{equ:sharp} \sum_{e\in E_1, e\geq 3}(e-2) + \sum_{e\in E_2, e\geq 3}(e-2) = u_1 + u_2 -2. \end{equation}
If $u_1+u_2 =2$, then by \eqref{equ:sharp} the orbits of the branch cycles $x_1,x_2$ over $P_1,P_2$ are of length $\leq 2$ and hence $x_1$ and $x_2$ are involutions. In this case $G$ is generated by two involutions, contradicting that $G$ is not dihedral.

If $u_1+u_2 =3$, then by \eqref{equ:sharp} there is a single $e_0\in E_1\cup E_2$ that is greater than $2$, and $e_0 = 3$. In this case the resulting ramification types with primitive\footnote{The primitivity condition is used merely in order to ensure that  for two distinct branch points $P,Q$ of $k(x_1)/k(t)$, there is no prime which divides all  entries in both $E_{k(x_1)/k(t)}(P)$ and $E_{k(x_1)/k(t)}(Q)$, see {\cite[Lemma 9.1]{NZ}}.} Galois group are (I2.3)-(I2.5) in Table \ref{table:NZ1}.  If $u_1+u_2 = 4$, then either there is a single $e_0\in E_1\cup E_2$ that is greater than $2$ and $e_0 = 4$, or there are exactly two $e_1,e_2\in E_1\cup E_2$ that are greater than $2$ and $e_1=e_2=3$. 
In the former case, the ramification types with primitive Galois group are (I2.6)-(I2.8) in Table \ref{table:NZ1}. 
Now assume that $e_{k(x_1)/k(t)}(Q_i)=3$, $i=1,2$ and $e_{k(x_1)/k(t)}(Q)\leq 2$ for any place $Q$ over $P_1$ or $P_2$. Applying the above  Abhyankar lemma shows that $Q_1,Q_2$ and the $u_1+u_2=4$ unramified places are all branch points of $F/k(x_1)$, giving more than $5$ branch points, proving the claim. 

To treat cases (I2.3)-(I2.8) from Table \ref{table:NZ1}, note that one of $E_{k(x_1)/k(t)}(P_i)$, $i=1,2$ has an entry which is larger than three. WLOG assume it is $P_1$. 
Then the above  Abhyankar lemma shows that each of the places $Q$ of $k(x_1)$ over $P_1$ with ramification index $2$ is a branch point of $F/k(x_1)$. For each of the types there are at least $(n-7)/2\geq 5$ such places, completing the proof of (1). 

For part (2), recall that the natural action of $G$ on $S=\{1,\ldots,n\}$ is equivalent to its action on $G/G_1$, and that the action of $G$ on $2$-sets from $S$ is equivalent to its action on $G/G_2$. Under this equivalence, for a place $P$ of $k(t)$  with branch cycle $x_P\in G$,   there is a one to one correspondence between the orbits $U$ of $x_P$ on $2$-sets and the places $Q_U$ of $\Omega^{G_2}$ lying over $P$. 

Given orbits $R_1, R_2, R_3\subseteq S$ of $x_P$ with lengths $r_1,r_2,r_3$, respectively, such that $r_1$ does not divide $\lcm(r_2,r_3)$, we claim that every place $Q_U$ of $\Omega^{G_2}$ lying over $P$, corresponding to an orbit $U\subseteq R_2\cup R_3$ on $2$-sets, is a branch point of $\Omega^{G_2}(x_1)/\Omega^{G_2}$. 
Note that since  $1\not\in G_2$, similarly to the above correspondence,  the orbits $\hat U$ of $x_P$ on pairs $(s,C)$ where $s\in S$ and $C\subseteq S$ is a $2$-set are in one to one correspondence with the places $Q_{\hat U}$ of $\Omega^{G_1\cap G_2} = \Omega^{G_2}(x_1)$. Moreover, the correspondence is picked so that $Q_{\hat U}$ lies over a place $Q_U$ (resp.~$Q_R$) of $\Omega^{G_2}$ (resp.~$k(x_1)$) if and only if $U$ (resp.~$R$) is the image of $\hat U$ under the projection $(s,C)\mapsto C$ (resp.~$(s,C)\mapsto s$). Hence, given a place $Q_{R_1}$ of $k(x_1)$ and a place $Q_U$ of $\Omega^{G_2}$ for $U\subseteq R_2\cup R_3$ with $R_1\neq R_2,R_3$. Since $r_1$ does not divide $\lcm(r_2,r_3)$, $R_1\neq R_2,R_3$ and hence for every orbit $U$ of $x_P$ acting on $2$-sets from $R_2\cup R_3$, there is a place $Q_{\hat U}$ of $\Omega^{G_2}(x_1)$ lying over $Q_U$ and $Q_{R_1}$. Now by Abhyankar's lemma $e(Q_{\hat U}/P) = \lcm(r_1,e)$ where $e:=e(Q_{U}/P)$, and $e$ divides  $\lcm(r_2,r_3)$. Thus $e(\hat Q_U/Q_U) = \lcm(r_1,e)/e = r_1/(r_1,e)>1$, and $Q_U$ is a branch point, proving the claim. 

For types (I1.1)-(I2.2), the branch cycle $x_3$ of the last branch point $P_3$ has an orbit $R_1$ of length $2$ which larger than $\lcm(r_2,r_3)$ for any two fixed points $R_2,R_3$ of $x_3$. Since there are $n-2$ such fixed points, we have at least $(n-2)(n-3)/2$ branch points. Similarly, for types (I2.3)-(I2.8), the branch cycle $x_2$ of the last branch point $P_2$ has an orbit $R_1$ of length $3$ or $4$ which is larger than $\lcm(r_2,r_3)=2$ for any two length $2$ orbits $R_2,R_3$ of $x_2$. Since there are at least $(n-7)/2$ length $2$ orbits of $x_2$, $\Omega^{G_2}(x_1)/\Omega^{G_2}$ has at least ${(n-7)/2 \choose 2}>5$ branch points. 
%

\bibliographystyle{plain}

\end{document}